\documentclass[a4paper,11pt,twoside]{amsart}
\usepackage[english]{babel}
\usepackage[utf8]{inputenc}

\usepackage[a4paper,inner=2.9cm,outer=2.9cm,top=4cm,bottom=4cm,pdftex]{geometry}
\usepackage{fancyhdr}
\usepackage{titlesec}
\usepackage{color}
\usepackage{bold-extra}
\usepackage{ mathrsfs }
\usepackage{enumitem}
\titleformat{\section}{\normalfont\scshape\centering}{\thesection}{1em}{}
  \titleformat{\subsection}{\bfseries}{\thesubsection}{1em}{}
  \titleformat{\subsubsection}{\bfseries}{\thesubsubsection}{1em}{}

\usepackage{bm}
\usepackage{comment}
\usepackage{graphics}
\usepackage{aliascnt}
\usepackage[pdftex,citecolor=green,linkcolor=red]{hyperref}

\usepackage{amsmath}
\usepackage{amsfonts}
\usepackage{amssymb}
\usepackage{amsthm}
\usepackage{comment}
\usepackage{mathtools}

\newtheorem{theorem}{Theorem}[section]
\newtheorem{corollary}[theorem]{Corollary}

\newtheorem{lemma}[theorem]{Lemma}
\newtheorem{proposition}[theorem]{Proposition}
\theoremstyle{definition}
\newtheorem{definition}[theorem]{Definition}
\newtheorem{remark}[theorem]{Remark}
\newtheorem{conjecture}[theorem]{Conjecture}

\numberwithin{equation}{section}

\newcommand\eps{\varepsilon}

\renewcommand{\Re}{\textnormal{Re}}

\newcommand\R{\mathbb{R}}
\newcommand\Z{\mathbb{Z}}
\newcommand\D{\mathbb{D}}

\newcommand\N{\mathbb{N}}
\newcommand\C{\mathbb{C}}

\newcommand\n{\mathbf{n}}

\newcommand{\1}{{\mathbf 1}}
\newcommand{\dd}{{\rm d}}

\begin{document}
\title{On the Hardy--Littlewood--Chowla conjecture\\on average}

\author{Jared Duker Lichtman}
\address{Mathematical Institute, University of Oxford \\
Woodstock Road \\
Oxford OX2 6GG \\
United Kingdom}
\email{jared.d.lichtman@gmail.com}

\author{Joni Ter\"av\"ainen}
\address{Department of Mathematics and Statistics \\
University of Turku, 20014 Turku\\
Finland}
\email{joni.p.teravainen@gmail.com}


%\date{\today}
\date{June 21, 2022.}

\maketitle

\begin{abstract}
There has been recent interest in a hybrid form of the celebrated conjectures of Hardy--Littlewood and of Chowla. We prove that for any $k,\ell\ge1$ and distinct integers $h_2,\ldots,h_k,a_1,\ldots,a_\ell$, we have
$$\sum_{n\leq X}\mu(n+h_1)\cdots \mu(n+h_k)\Lambda(n+a_1)\cdots\Lambda(n+a_{\ell})=o(X)$$
for all except $o(H)$ values of $h_1\leq H$, so long as $H\geq (\log X)^{\ell+\eps}$. This improves on the range $H\ge (\log X)^{\psi(X)}$, $\psi(X)\to\infty$, obtained in previous work of the first author. Our results also generalize from the M\"obius function $\mu$ to arbitrary (non-pretentious) multiplicative functions.
\end{abstract}

\section{Introduction}
Let $\mu$ and $\Lambda$ denote the M\"obius and von Mangoldt functions, respectively, defined by
\[
\mu(n) = \begin{cases}
(-1)^r & n=p_1\cdots p_r,\; p_i \text{ distinct,}\\
0 & \ \text{else.}
\end{cases}\quad\text{and} \quad
\Lambda(n) = \begin{cases}
\log p & n=p^e,\;e\ge1,\\
0 & \ \text{else.}
\end{cases}
\]
Also let $\lambda$ be the Liouville function given by $\lambda(n)=(-1)^{\Omega(n)}$, where $\Omega(n)$ is the total number of prime factors of $n$ with multiplicities. 
For technical convenience, we extend these functions to the nonpositive integers as zero (the choice of the extension makes no difference).  Recall that the prime number theorem is equivalent to the average bounds $\sum_{n\le X}\mu(n) = o(X)$ and $\sum_{n\le X}\Lambda(n) = (1+o(1))X$. 

The influential conjectures of Chowla~\cite{Chowla} and Hardy--Littlewood~\cite{HardLittl} assert that for any fixed tuple of distinct integers $h_1,\ldots,h_k$,
\begin{align*}
\sum_{n\le X}\mu(n+h_1)\cdots\mu(n+h_k) \ &= \ o(X),\\
\sum_{n\le X}\Lambda(n+h_1)\cdots\Lambda(n+h_k) \ & = \ (\mathfrak{S}(\mathcal{H}) + o(1))X,
\end{align*}
where $\mathcal{H}=\{h_1,\ldots, h_k\}$ and the singular series is given by
\begin{align}\label{eq:S}\mathfrak{S}(\mathcal{H}) = \prod_{p}\frac{1-\nu_p(\mathcal{H})/p}{(1-1/p)^k},
\end{align}
where $\nu_p(\mathcal{H}) = |\{h\;(\text{mod }p) : h\in\mathcal{H}\}|$. Both conjectures remain open for any $k\ge2$.

It is natural to consider the following hybrid conjecture which, following~\cite{Lichtman}, \cite{tt-siegel}, we call the Hardy--Littlewood--Chowla conjecture.

\begin{conjecture}[Hardy--Littlewood--Chowla]\label{conj:HLC}
Let $k,\ell\geq 0$, and let $h_1,\ldots,h_k$, $a_1,\ldots, a_\ell$ be distinct integers. Then we have\footnote{Here we use the convention that an empty product equals to $1$.}
\begin{align}\label{eq:HLCconj}
\sum_{n\le X}\mu(n+h_1)\cdots\mu(n+h_k)\Lambda(n+a_1)\cdots\Lambda(n+a_\ell) \ & = \ (\mathfrak{S} + o(1))X.
\end{align}
Here $\mathfrak{S}=\mathfrak{S}(\{a_1,\ldots, a_\ell\})$ as in \eqref{eq:S} if $k=0$, and $\mathfrak{S}=0$ if $k>0$.
\end{conjecture}

The `pure' cases $k=0$ and $\ell=0$ specialize to the original conjectures of Hardy--Littlewood and of Chowla, respectively. We remark that under the hypothetical assumption of infinitely many Siegel zeros, Conjecture~\ref{conj:HLC} was recently verified for $\ell\leq 2$ by Tao and the second author~\cite{tt-siegel} (generalizing works of Heath-Brown~\cite{hb} and Chinis~\cite{chinis}). In the current paper, we will be concerned with unconditional results. 

Our main result is an averaged form of Conjecture~\ref{conj:HLC} for the genuinely `hybrid' cases $k,\ell\ge1$.

\begin{theorem}[Hardy--Littlewood--Chowla on average]\label{thm_hybrid} Let $\eps>0$ and $k,\ell\geq 1$ be fixed, and let $h_2,\ldots ,h_{k}$, $a_1,\ldots ,a_{\ell}$ be fixed and distinct positive integers. 
\begin{enumerate}
\item[(i)] Let $(\log X)^{\ell+\eps}\leq H\leq \exp((\log X)^{a})$ with $a=a(\varepsilon,\ell)>0$ small enough. Then
\begin{align}\label{eq:hybrid}
\sum_{h_1\le H}\bigg|\sum_{n\leq X}\mu(n+h_1)\cdots \mu(n+h_k)\Lambda(n+a_1)\cdots\Lambda(n+a_{\ell})\bigg| \ \ll \ HX\,\frac{\log \log H}{\log H}.
\end{align}
\item[(ii)] Let $C\geq 1$ be fixed, and let $(\log X)^{\ell+\eps}\leq H\leq (\log X)^{C}$.  There exists an absolute constant $c>0$ such that for any $10^4\ell\varepsilon^{-1}(\log \log H)/(\log H)\leq \delta\leq c/C$, we have 
\begin{align*}
\bigg|\sum_{n\leq X}\mu(n+h_1)\cdots \mu(n+h_k)\Lambda(n+a_1)\cdots\Lambda(n+a_{\ell})\bigg|\leq \delta X    
\end{align*}
for all except $\ll H^{1-c\delta \varepsilon/\ell}$ values of $h_1\leq H$. 
\end{enumerate}
\end{theorem}

In particular, \eqref{eq:HLCconj} holds for all but $o(H)$ values of $h_1\leq H$ (and the exceptional set can be taken to be nearly power-saving). 

As is clear from the proof, Theorem~\ref{thm_hybrid} holds equally well with the Liouville function in place of the M\"obius function. 

\begin{remark} It is likely that with some extra effort (in particular, refining Proposition~\ref{thm:MRTChow} for $\mu$ in the spirit of~\cite[Theorem 2.2]{Lichtman}) Theorem~\ref{thm_hybrid}(ii) could be extended to the regime $(\log X)^{\ell+\varepsilon}\leq H\leq X$. However, our main interest here is in taking $H$ as small as possible.  
\end{remark}

Earlier, Matom\"aki, Radziwi\l{}\l{} and Tao~\cite{MRTChow} proved the case $\ell=0$ of Conjecture~\ref{conj:HLC} for almost all $h_1\leq H$, with $H=\psi(X)$ tending to infinity arbitrarily slowly. In~\cite{MRTCor}, they considered the case $k=0$, $\ell=2$ and proved the conjecture for almost all $h_1\leq H$, with $H\geq X^{8/33+\varepsilon}$. The first author~\cite{Lichtman} obtained cancellation in the range $H\geq (\log X)^{\psi(X)}$, implied by the stronger quantitative bound
\begin{align*}
\sum_{h_1,\ldots, h_k\leq H}\bigg|\sum_{n\leq X} \mu(n+h_1)\cdots \mu(n+h_k) & \Lambda(n+a_1)\cdots \Lambda(n+a_{\ell}) \bigg| \\
&\ll \  \frac{H^kX}{\min\big\{\psi(X)^k,(\log X)^{k/3-o(1)}\big\}}.
\end{align*}

The simplest hybrid case $k=\ell=1$ reduces to the M\"obius function on shifted primes. It is a folklore conjecture and a well-known model case for the parity problem in sieve theory  that 
\begin{align*}
 \sum_{p\le X}\mu(p+h) \ = \ o(\pi(X))    
\end{align*}
for any fixed shift $h\neq 0$. This has appeared in print (for $h=1$) at least since Hildebrand~\cite{hildebrand-1989}; see also Pintz~\cite{Pintz} and  Murty--Vatwani~\cite[(1.2)]{MurtyV}. Theorem~\ref{thm_hybrid} directly implies an averaged form of this conjecture for $H\ge (\log X)^{1+\eps}$.

\begin{corollary}\label{cor:shifted} Let $\eps>0$. Then, for $X\geq H\geq (\log X)^{1+\eps}$, we have
\begin{align}\label{eq:musum}
\sum_{p\leq X}\mu(p+h)=o(\pi(X))
\end{align}
for all except $o(H)$ values of $h\leq H$.
\end{corollary}

This improves on work of the first author~\cite{Lichtman} that established Corollary~\ref{cor:shifted} when $H\geq (\log X)^{\psi(X)}$ for any function $\psi(X)$ tending to infinity with $X$.

\subsection{Other multiplicative functions}\label{subsec:mult}

We also consider a variant of Conjecture~\ref{conj:HLC} where the occurrences of the M\"obius function are replaced with other 1-bounded multiplicative functions $f_i:\N\to\C$, with $f_1$ not pretending to be a twisted character $\chi(n)n^{it}$ for any $\chi \pmod q$ and $t\in \R$.  Here, following Granville and Soundararajan~\cite{GranSound}, pretentiousness is measured by the pretentious distance
\begin{align*}
\D(f,g;X) = \bigg(\sum_{p\le X}\frac{1-\Re(f(p)\overline{g(p)})}{p}\bigg)^{1/2}
\end{align*}
and the related quantity
\begin{align}
M(f;X,Q) = \inf_{\substack{|t|\le X\\\chi\,(q), \, q\le Q}} \D\big(f,\, n\mapsto n^{it}\chi(n); X\big)^2.
\end{align}
The following generalization of Conjecture~\ref{conj:HLC} is closely related to Elliott's conjecture~\cite{elliott} on correlations of multiplicative functions (in fact, the case $\ell=0$ is precisely Elliott's conjecture), so we shall call it the Hardy--Littlewood--Elliott conjecture.

\begin{conjecture}[Hardy--Littlewood--Elliott]\label{conj:elliott} Let $h_1,\ldots, h_k,a_1,\ldots, a_{\ell}$ be distinct integers. Fix 1-bounded multiplicative functions $f_1,\ldots, f_k$. Suppose that $f_1$ is non-pretentious in the sense that
\begin{align*}
M(f_1;X,Q) \xrightarrow{X\to \infty} \infty  
\end{align*}
for any $Q\geq 1$. Then we have
\begin{align}\label{eq:elliott}
\sum_{n\leq X}f_1(n+h_1)\cdots f_k(n+h_k)\Lambda(n+a_1)\cdots \Lambda(n+a_{\ell}) =o(X).   
\end{align}
\end{conjecture}
The case $\ell=0$ of this is Elliott's conjecture (in the slightly corrected formulation of~\cite[Conjecture 1.3]{tt-ant}).

We extend our results to non-pretentious multiplicative functions in the regime $H\ge (\log X)^{\ell+\varepsilon}$.

\begin{theorem}[Hardy--Littlewood--Elliott on average]\label{thm:multsums}
Let $\eps>0$, $k,\ell\geq 1$ be fixed, and let $h_2,\ldots, h_k,a_1,\ldots, a_{\ell}$ be fixed and distinct positive integers. Let $(\log X)^{\ell+\varepsilon}\leq H\leq \exp((\log X)^{1/1000})$. Let  $f_1,\ldots, f_k$ be  1-bounded multiplicative functions. Then we have 
\begin{align*}
\sum_{h_1\leq H} &\bigg|\sum_{n\leq X}f_1(n+h_1)\cdots f_k(n+h_k)\Lambda(n+a_1)\cdots \Lambda(n+a_{\ell})\bigg|\\
&\ll HX\left(\exp\big({-M(f_1;X,\min\{(\log X)^{1/125},(\log H)^5\}\big)}/(10000\ell\varepsilon^{-1})\big)+\frac{\log \log H}{\log H}\right).    
\end{align*}
In particular, if $(\log X)^{\ell+\eps} \leq H\leq  \exp((\log X)^{1/1000})$ and
\begin{align*}
 M(f_1;X,\min\{(\log X)^{1/125},(\log H)^5\})\xrightarrow{X\to \infty} \infty,
 \end{align*}
then \eqref{eq:elliott} holds for all except $o(H)$ values of $h_1\leq H$. 
\end{theorem}

\begin{remark}\label{rem:quant}
As a well-known consequence of the Vinogradov--Korobov zero-free region for $L$-functions, for any fixed $\epsilon>0$, $A\geq 1$ we have
\begin{align}\label{eq:dist}\begin{split}
M(\mu; X,(\log X)^{A}) \ &\ge \ \inf_{\substack{|t|\le X\\\chi\,(q), \, q\le Q}}\,\, \sum_{\exp((\log X)^{2/3+\epsilon})\le p\le X}\frac{1+\Re(\chi(p)p^{it})}{p}\\
\ &\ge \ \Big(\frac{1}{3}-\epsilon\Big)\log\log X + O(1).\end{split}
\end{align}
Hence, the bound of Theorem~\ref{thm:multsums} in the case of $f_1=\mu$ simplifies to
\begin{align*}
\sum_{h_1\leq H}\bigg|\sum_{n\leq X}\mu(n+h_1)\cdots f_k(n+h_k)&\Lambda(n+a_1)\cdots \Lambda(n+a_{\ell})\bigg|\\
&\ll HX\left(\frac{\log \log H}{\log H}+(\log X)^{-a}\right)    
\end{align*}
for some constant $a=a(\varepsilon,\ell)>0$. The same holds with $\lambda$ in place of $\mu$, since \eqref{eq:dist} holds equally well for $\lambda$. Restricting to $\exp((\log X)^{a})\geq H\geq (\log X)^{\ell+\eps}$, the bound simplifies to $\ll HX(\log \log H)/\log H$. Hence, Theorem~\ref{thm_hybrid}(i) is a special case of Theorem~\ref{thm:multsums}. The rest of the paper is therefore devoted to the proofs of Theorem~\ref{thm:multsums} and Theorem~\ref{thm_hybrid}(ii).
\end{remark}

Theorem~\ref{thm:multsums} improves on the range $H\geq \exp((\log X)^{5/8+\varepsilon})$ which follows (under a slightly different pretentiousness hypothesis) from Fourier uniformity bounds of Matom\"aki, Radziwi\l\l, Tao, Ziegler and the second author~\cite[Theorem 1.8]{MRTTZ}. See~\cite[Theorem 1.8]{Lichtman} for the details of this implication\footnote{\cite[Theorem 1.8]{Lichtman} relies on~\cite[Theorem 1.4]{MRTUnif} as the Fourier uniformity input, and hence produces a result for $H=X^{\theta}$. Inserting instead~\cite[Theorem 1.12]{MRTTZ} into the argument produces a result for $H=\exp((\log X)^{5/8+\varepsilon})$.}. We also note that the case $\ell=0$ of Conjecture~\ref{conj:elliott} was proven on average by Matom\"aki--Radziwi\l{}\l{}--Tao~\cite{MRTChow} in the regime $H=\psi(X)\to \infty$.

\subsection{Outline of the proof}

The proof method of Theorem~\ref{thm_hybrid} (and Theorem~\ref{thm:multsums}) is somewhat different from that of~\cite{Lichtman} (or~\cite{MRTChow}). Both proofs begin with Fourier-analytic identities, namely~\cite[Lemma 2.1]{Lichtman} and Proposition~\ref{thm:fourier} below. The latter allows one to study mixed 2-point correlations on average, for some functions $f$ and $g$ which satisfy certain technical hypotheses. Crucially, in Proposition~\ref{thm:fourier} having qualitative cancellation for the exponential sum of $f$ on average over short intervals $[x,x+H]$ suffices, whereas in~\cite[Lemma 2.1]{Lichtman} one needs to save a factor comparable to $((1/X)\sum_{n\leq X}|g(n)|^2)^{-1/2}$ (but no additional hypotheses were required of $f, g$). For intervals of length $H=(\log X)^{O(1)}$, the bound in~\cite{MRTChow} for the exponential sum  of a multiplicative $f$ on average over short intervals $[x,x+H]$ saves less than $(\log X)^{-o(1)}$, which is problematic if the mean value of $|g|^2$ is a power of $\log X$, so the weakening of the Fourier assumption on $f$ in Proposition~\ref{thm:fourier} is important for us.

One ultimately applies Proposition~\ref{thm:fourier} for the choices $f = \mu$ (say) and $g(n) = \mu(n + h_2)\cdots \mu(n+h_k)\Lambda(n+a_1)\cdots \Lambda(n+a_\ell)$, in which case verifying the technical hypotheses may be reduced to the following problems:

\begin{enumerate}
\item[(I)] cancellation in short exponential sums of $\mu$ on average,
\item[(II)] moment estimates for short exponential sums over primes, i.e. for `typical' $x$,
\begin{align*}
\int_{0}^1\bigg|\sum_{x\leq n\leq x+H}\Lambda(n)e(\alpha n)\bigg|^{2m}\dd{\alpha} \ \ll \ H^{2m-1}.
\end{align*}
\end{enumerate}

Here (I) is deduced from the work of Matom\"aki--Radziwi\l{}\l{}--Tao~\cite{MRTChow}, with special care paid to the shape of the error terms in order to obtain a power-saving error term in Theorem~\ref{thm_hybrid}(ii). This is carried out in Proposition~\ref{thm:MRTChow}.

For (II), the intervals involved are too short for sieve methods to provide us with a strong enough pointwise bound on these moments. However, on average over $x$, in Proposition~\ref{lem:Lambda6moment} we show that the moments above are of the expected order of magnitude by applying upper bounds for prime tuples coming from Selberg's sieve, and by performing an analysis of the resulting singular series.

The method presented is rather flexible, and could be applied (with some additional effort) to other correlation problems as well, such as the sums
\begin{align*}
 \sum_{n\leq X}\mu(n+h)\Lambda(P(n))\quad    \textnormal{or} \quad \sum_{n\leq X}\mu(n+h)\1_{\exists a,b\in\mathbb{Z} \; :\; n=a^2+b^2}
\end{align*}
on average over $h$ (with $P(Y)\in \mathbb{Z}[Y]$). We leave the details of such generalizations to the interested reader.

\subsection{Acknowledgments}

The first author was supported by a Clarendon Scholarship.
The second author was supported by a Titchmarsh Fellowship and Academy of Finland grant no. 340098. The authors would like to thank the anonymous referee for detailed feedback and corrections.

\section{Auxiliary results}

\subsection{A short exponential sum estimate}

As in~\cite{MRTChow}, we introduce a set $\mathcal{S}_{P_1,Q_1,X}$ (depending on parameters $10<P_1<Q_1\leq X$) consisting of the ``typical'' numbers having prime factors from certain long ranges. More precisely, $\mathcal{S}_{P_1,Q_1,X}$ is the set of positive integers having a prime factor in each of the ranges $[P_j,Q_j]$, $j=1,\ldots, J$, where for $j>1$ we define
\begin{align*}
P_j &=\exp(j^{4j}(\log Q_1)^{j-1}(\log P_1)),\\
Q_j &=\exp(j^{4j+2}(\log Q_1)^j),   
\end{align*}
and $J$ is the largest index such that $Q_J\leq \exp(\sqrt{\frac{1}{2}\log X})$. If $\mathcal{S}_{P_1,Q_1,X}^{c}$ denotes the complement of $\mathcal{S}_{P_1,Q_1,X}$ in $\mathbb{N}\cap [1,X]$, then the fundamental lemma of sieve theory~\cite[Lemma 6.17]{Opera} tells us that
\begin{align}\label{eq_Sbound}
\big|\mathcal{S}_{P_1,Q_1,X}^{c}\big|\ll \sum_{j\leq J}\frac{\log P_j}{\log Q_j}X\ll \frac{\log P_1}{\log Q_1}X.    
\end{align}

We then have the following exponential sum estimate (with the supremum outside) for short sums of multiplicative functions, which follows from~\cite[Theorem 2.3]{MRTChow}. 

\begin{proposition} \label{thm:MRTChow}
Let $3\leq H\leq \exp((\log X)^{1/1000})$ and
\begin{align*}
\frac{\log\log H}{\log H} \leq \ \delta \ \leq \frac{1}{2000}.
\end{align*}
Let $f$ be a 1-bounded multiplicative function. Then there exist $10<P_1<Q_1\leq X$ such that 
\begin{align}\label{eq:P1Q1}
\frac{\log P_1}{\log Q_1}\ll \delta +\exp(-M(f;X,Q)/2000)
\end{align}
and such that for $\mathcal{S}=\mathcal{S}_{P_1,Q_1,X}$ we have
\begin{align}\label{eq:fourier}
\sup_{\alpha}\int_0^X & \bigg|\sum_{x\le n\le x+H}f(n)\1_{\mathcal{S}}(n) e(n\alpha)\bigg|\dd{x}\ll \bigg(\exp(-M(f;X,Q)/2000)+H^{-\delta}\bigg)HX,
\end{align}
where $Q= \min\{(\log X)^{1/125}, (\log H)^5\}$.
\end{proposition}

\begin{remark} 
Applying \eqref{eq:dist}, and assuming that $H\leq (\log X)^C$, the bound \eqref{eq:fourier} in the case of $f=\mu$ simplifies to
\begin{align}\label{eq:fourierb}
\sup_{\alpha}\int_0^X & \bigg|\sum_{x\le n\le x+H}\mu(n)\1_{\mathcal{S}}(n) e(n\alpha)\bigg|\dd{x}\ll H^{1-\delta}X
\end{align}
for some $\delta=\delta(C)>0$, and the same holds with $\lambda$ in place of $\mu$. Moreover, \eqref{eq:P1Q1} then simplifies to $(\log P_1)/(\log Q_1)\ll \delta$.  
\end{remark}

\begin{proof}
Without loss of generality, we may assume that $H\geq H_0$ for a large enough constant $H_0$.
Denote the integral in \eqref{eq:fourier} by $J_f$. Suppose first that
\begin{align*}
(\log H)^{6}\leq \exp(M(f;X,Q)/300).     
\end{align*}
Let
\begin{align*}
W=\min\{H^{6\delta},\exp(M(f;X,Q)/300)\}, 
\end{align*}
and let $P_1=W^{200}$, $Q_1=H/W^{3}$ be as in~\cite[Theorem 2.3]{MRTChow}. Note that $W\ll (\log X)^{1/150}$ by the fact that $M(f;X,Q)\leq 2\log \log X+O(1)$. Note also that $\frac{\log P_1}{\log Q_1}\ll \frac{\log W}{\log H}\ll \delta$, and that all the conditions for $W$ in~\cite[Theorem 2.3]{MRTChow} are satisfied (that is, $W\geq (\log H)^5$ and $W\leq \min\{H^{1/250},(\log X)^{1/125}\}$ and $W\leq \exp(M(f;X,Q)/3)$), so we obtain
\begin{align*}
J_f \ll \frac{(\log H)^{1/4}(\log \log H)}{W^{1/4}}HX\ll (H^{-\delta}+\exp(-M(f;X,Q)/2000))HX.     
\end{align*}

Suppose then that
\begin{align*}
\exp(M(f;X,Q)/300):=(\log H_1)^6<(\log H)^6.     
\end{align*}
Let
\begin{align}
W=(\log H_1)^6=\exp(M(f;X,Q)/300),    
\end{align}
and let $P_1=W^{200}$, $Q_1=H/W^3$ as before. 
We may assume $H_1\geq 1000$, since otherwise the claim is trivial. Now, by noting that $\frac{\log P_1}{\log Q_1}\ll \exp(-M(f;X,Q)/2000)$ and splitting the sum on $H$ in \eqref{eq:fourier} into shorter sums of length in $[H_1/4,H_1/2]$, it suffices to prove that
\begin{align*}
\sup_{\alpha}\int_0^X & \bigg|\sum_{x\le n\le x+H'}f(n)\1_{\mathcal{S}}(n) e(n\alpha)\bigg|\dd{x}\ll \exp(-M(f;X,Q)/2000)H'X,
\end{align*}
uniformly for $H'\in [H_1/4,H_1/2]$. Since the conditions for $W$ in~\cite[Theorem 2.3]{MRTChow} are satisfied, we can again apply that theorem to obtain the desired bound (noting that $\frac{(\log H_1)^{1/4}(\log \log H_1)}{W^{1/4}}\ll \exp(-M(f;X,Q)/2000)$ in this situation). This completes the proof. 
\end{proof}

\subsection{Upper-bounding correlations of primes}

The goal of this subsection is to prove  Proposition~\ref{lem:Lambda6moment} below, which gives optimal bounds for even moments of the short exponential sums associated with correlations of the von Mangoldt function. We first need a few lemmas on tuples of primes and averages of singular series.  

Recall for a tuple $\mathcal H$ of integers, $\nu_{p}(\mathcal{H}) = |\{h\pmod{p} : h\in\mathcal H\}|$. A well-known application of Selberg's sieve upper bounds the number of $k$-tuples of primes, for fixed $k$.

\begin{lemma} \label{Selbtuple}
Let $k\geq 1$ be fixed, let $X\geq X_0(k)$, and suppose that $h_1,\ldots, h_k\in[-X,X]$ are distinct integers. Denote $\mathcal H = \{h_1,\ldots,h_k\}$. Then we have
\begin{align*}
\big|\{n\le X \; : \; n+h_1,\ldots, n+h_k\in\mathbb P\}\big| \ \le \ k!\cdot 2^k\, \frac{\mathfrak{S}(\mathcal H)X}{(\log X)^k}\Big(1 + O\Big(\frac{\log\log X}{\log X}\Big)\Big),
\end{align*}
where the singular series is given by
\begin{align*}
\mathfrak{S}(\mathcal H) := \prod_p \Big(1-\frac{\nu_{p}(\mathcal{H})}{p}\Big)\Big(1-\frac{1}{p}\Big)^{-k}.
\end{align*}
\end{lemma}

\begin{proof}
This is~\cite[Theorem 7.16]{Opera}.
\end{proof}

We have a trivial upper bound for the values of the singular series, which we will need in what follows. 

\begin{lemma}\label{Singularbound}
Let $k\geq 1$, and let $\mathcal{H}=\{h_1,\ldots, h_k\}$ be a tuple of distinct integers. Then 
\begin{align*}
\mathfrak{S}(\mathcal{H})\ll_k \prod_{1\leq i<j\leq k}\Big(\frac{|h_i-h_j|}{\varphi(|h_i-h_j|)}\Big)^k.
\end{align*}
\end{lemma}

\begin{proof}
Note that $\nu_{p}(\mathcal{H})\le k$, and that if $\nu_{p}(\mathcal{H})<k$ then $p$ must divide $h_i-h_j$ for some $i<j$. Hence
\begin{align*}
\mathfrak{S}(\mathcal{H}) &\le \prod_{\nu_{p}(\mathcal{H})=k} \left(1-\frac{k}{p}\right)\left(1-\frac{1}{p}\right)^{-k}\prod_{\nu_{p}(\mathcal{H})<k}\left(1-\frac{1}{p}\right)^{-k}\\
&\le \prod_{p>k}\left(1-\frac{k}{p}\right)\left(1-\frac{1}{p}\right)^{-k}\prod_{1\leq i<j\leq k}\prod_{p\mid h_i-h_j}\left(1-\frac{1}{p}\right)^{-k}\ll_k \prod_{1\leq i<j\leq k}\Big(\frac{|h_i-h_j|}{\varphi(|h_i-h_j|)}\Big)^k, 
\end{align*}
as claimed.
\end{proof}

We will also need an upper bound for the autocorrelations of $h/\varphi(h)$. 

\begin{lemma}\label{lemma_eulerphi}
Let $H\geq 1$ and let $C,k,r\geq 1$ be fixed. Then we have\footnote{In what follows, we interpret $0/\varphi(0)$ as $0$, say.}
\begin{align*}
\sum_{h\leq H}\prod_{i=1}^r\left(\frac{|L_i(h)|}{\varphi(|L_i(h)|)}\right)^k
\ll_{r,k,C} H
\end{align*}
uniformly for any linear functions $L_i(h)=a_ih+b_i$ with $a_i,b_i\in [-C,C]\cap\Z$.
\end{lemma}

\begin{proof} First note that for $H'\geq 1$ and any integer $m\geq 1$, we have
\begin{align}\label{eq:heulerphi}
\sum_{h\leq H'}\left(\frac{h}{\varphi(h)}\right)^{m} \ &= \ \sum_{h\leq H'} \bigg(\sum_{d\mid h}\frac{1}{d}\bigg)^{m} =\sum_{d_1,\ldots, d_m\geq 1} \frac{1}{d_1\cdots d_m}\sum_{\substack{h\leq H'\\d_i\mid h\;\forall i\le m}}1 \nonumber\\
&\ll H'\sum_{d_1,\ldots, d_m\geq 1}\frac{1}{d_1\cdots d_m[d_1,\ldots, d_m]} \nonumber\\
&\ll H'\prod_{p_1,\ldots, p_m}\left(1+\sum_{\substack{j_1,\ldots, j_m\geq 0\\(j_1,\ldots, j_m)\neq (0,\ldots, 0)}}\frac{1}{p_1^{j_1}\cdots p_m^{j_m}[p_1^{j_1},\ldots, p_m^{j_m}]}\right) \ \ll_m \ H'.
\end{align}
Then, by H\"older's inequality and \eqref{eq:heulerphi} with $H'=(C+1)H$, we have
\begin{align*}
\sum_{h\leq H}\prod_{i\le r}\left(\frac{|L_i(h)|}{\varphi(|L_i(h)|)}\right)^k
&\leq \prod_{i\le r}  \left(\sum_{h\leq H} \left(\frac{|L_i(h)|}{\varphi(|L_i(h)|)}\right)^{rk}\right)^{1/r}\\
&\leq \prod_{i\le r}  \left(\sum_{h\leq (C+1)H} \left(\frac{h}{\varphi(h)}\right)^{rk}\right)^{1/r} \ \ll_{r,k,C} \ H,
\end{align*}
using the fact that $L_i(\cdot)$ takes any integer value at most once.
\end{proof}

The previous lemmas lead to the following corollary.

\begin{corollary}\label{cor:singular} Let $k\geq m\geq 1$ be fixed. Then 
\begin{align*}
\sum_{\mathbf{h}\in [0,H]^m}\sup_{L_1,\ldots,L_k}\mathfrak{S}(\{L_1(\mathbf{h}),\ldots, L_k(\mathbf{h})\}) \ \ll_{k,A} \ H^m,
\end{align*}
where the supremum ranges over affine-linear forms $L_j(h_1,\ldots,h_m) = a_j+\sum_{i=1}^ma_{i,j}h_i$ with integer coefficients and with $|a_{i,j}|,|a_j|\le A$.
\end{corollary}

\begin{proof} Denote by $\Sigma$ the sum of interest on left hand side. Since there are $O_{k,A}(1)$ many choices of affine-linear forms $L_1,\ldots,L_k$ with $A$-bounded coefficients, it suffices to take the supremum outside the sum in $\Sigma$. Then for each $L_1,\ldots,L_k$, by Lemma~\ref{Singularbound} we bound $\Sigma$ as
\begin{align*}
\Sigma \ &\ll_k \ \sum_{\mathbf{h}\in [0,H]^m}\prod_{1\le j<j'\le k}\left(\frac{|L_j(\mathbf{h})-L_{j'}(\mathbf{h}) |}{\varphi(|L_j(\mathbf{h})-L_{j'}(\mathbf{h})|)}\right)^k\\
&= \sum_{\mathbf{h}\in [0,H]^{m-1}}\sum_{h\le H}\prod_{1\le j<j'\le k}\left(\frac{|L_{j,j',\mathbf{h}}(h)|}{\varphi(|L_{j,j',\mathbf{h}}(h))|}\right)^k \ \ll \ \sum_{\mathbf{h}\in [0,H]^{m-1}}H \ll H^m,
\end{align*}
since to each $\mathbf{h}\in [0,H]^{m-1}$, we can apply Lemma~\ref{lemma_eulerphi} with $r=\binom{k}{2}$, $C = mA$, and linear functions $L_{j,j',\mathbf{h}}(h) := L_j(\mathbf{h},h)-L_{j'}(\mathbf{h},h)$.
\end{proof}

Using Corollary~\ref{cor:singular}, we can prove the following bound for high moments of short exponential sums associated with the correlations of the von Mangoldt function. This estimate will be needed in Section~\ref{sec:main}.

\begin{proposition}\label{lem:Lambda6moment}
Let $\ell\geq 1$, $k\geq 2$ be fixed, and let $a_1,\ldots,a_{\ell}$ be distinct fixed integers. For $X\geq H\geq 2$ and $|g(n)|\leq 1$, we have
\begin{align*}
M_{2k}:= &\int_{0}^X \int_{0}^{1}\bigg|\sum_{x\leq n\leq x+H} g(n) e(\alpha n)\prod_{j=1}^{\ell}\Lambda(n+a_j)\bigg|^{2k}\dd\alpha \dd x \\
& \ll \ X(H^{2k-1}+H^{k}(\log X)^{k\ell}+H(\log X)^{(2k-1)\ell}).
\end{align*}
In particular if $H\geq (\log X)^{\ell k/(k-1)}$ then
\begin{align}
M_{2k} \; \ll \; XH^{2k-1}.
\end{align}
\end{proposition}

\begin{remark}
The bound above for $M_{2k}$ is optimal up to a constant factor, assuming the Hardy--Littlewood prime tuples conjecture. Indeed, one can show that $M_{2k}\gg X H^{2k-1}$ in the case $g(n)\equiv 1$, assuming the Hardy--Littlewood conjecture (this is done by considering the contribution to \eqref{eq4.5} below coming from those $(n_1,\ldots, n_{2k})$ with $H>|n_i-n_j|>\max_{i}|a_i|$ for all $i\neq j$). Similarly, one can show $M_{2k}\gg XH^{k}(\log X)^{k\ell}$ (by considering the contribution to \eqref{eq4.5} below coming from those $(n_1,\ldots, n_{2k})$ with $n_{k+i}=n_i$ for all $i\le k$ and $H>|n_i-n_j|>\max_{i}|a_i|$). Lastly, one can show $M_{2k}\gg XH(\log X)^{(2k-1)\ell}$ (by considering the contribution to \eqref{eq4.5} below coming from those $(n_1,\ldots, n_{2k})$ with $n_1=\cdots=n_{2k}$).
\end{remark}

\begin{proof} Expanding out the definition of $M_{2k}$ using orthogonality, we obtain
\begin{align}\label{eq4.5}
M_{2k} &= \int_{0}^X\sum_{\substack{x\leq n_1,\ldots, n_{2k}\leq x+H\\n_1+\cdots +n_k=n_{k+1}+\cdots +n_{2k}}}g(n_1)\cdots g(n_k)\overline{g(n_{k+1})\cdots g(n_{2k})}\prod_{\substack{j\le \ell\\i\le 2k}} \Lambda(n_i+a_j) \dd x,
\end{align}
and so $|g(n)|\le 1$ implies
\begin{align}\label{eq5}
M_{2k} &\le \sum_{\substack{n_1,\ldots, n_{2k}\leq X+H\\n_1+\cdots +n_k=n_{k+1}+\cdots +n_{2k}}}\prod_{\substack{j\le \ell\\i\le 2k}} \Lambda(n_i+a_j)\cdot \int_{0}^X \1_{x\leq n_1\leq x+H}\cdots \1_{x\leq n_{2k}\leq x+H}\dd x.
\end{align}
For each $n_1,\ldots, n_{2k}\leq X+H$ denote $M=\max_{1\le i,j\le 2k} |n_i-n_j|$, so that
\begin{align*}
\int_{0}^X \1_{x\leq n_1\leq x+H}\cdots \1_{x\leq n_{2k}\leq x+H}\dd{x}&=(H-M)\1_{0\leq M\leq H} \ \leq \ H \1_{0\leq M< H},
\end{align*}
and if we let $n=n_{2k}$, $h_i = n_i-n$ (so $h_{2k}=0$), then \eqref{eq5} becomes
\begin{align*}
M_{2k} &\le \ H\sum_{\substack{n_1,\ldots, n_{2k}\leq X+H\\n_1+\cdots +n_k=n_{k+1}+\cdots +n_{2k}\\\max_{i,j}|n_i-n_j|< H}}\prod_{\substack{j\le \ell\\i\le 2k}} \Lambda(n_i+a_j)\\
&\leq \ H\sum_{\substack{-H\leq h_1,\ldots, h_{2k-1}\leq H\\h_1+\cdots+h_k=h_{k+1}+\cdots +h_{2k-1}}}\sum_{n\leq X+H}\prod_{\substack{1\le j\le \ell\\0\le i\le 2k-1}} \Lambda(n+h_i+a_j).
\end{align*}
Write $\mathcal{A}=\{a_1,\ldots, a_{\ell}\}$ and $\mathcal H=\{0,h_1,\ldots, h_{2k-1}\}$  (note $h_{2k-1}=h_1+\cdots+h_{k-1}-h_{k}-\cdots -h_{2k-2}$). Let $\mathcal{A}+\mathcal{H}$ denote the sumset, and observe that $|\mathcal{A}+\mathcal{H}|\in [\ell,2k\ell]$. In what follows, for convenience of notation we denote $h_0=0$. 

Now we partition the tuples $\mathcal H$ according to $|\mathcal{A}+\mathcal{H}|$, namely, for $m\in[0,(2k-1)\ell]$ denote
\begin{align*}
{\rm H}_m = \big\{(h_1,\ldots,h_{2k-2}) \; : \; |h_i|\le H, \ |\mathcal{A}+\mathcal{H}|=2k\ell-m \big\},
\end{align*}
so that by Lemma~\ref{Selbtuple},
\begin{align}\label{m6}
M_{2k} &\ll XH\sum_{m=0}^{(2k-1)\ell}(\log X)^{m}\sum_{(h_1,\ldots,h_{2k-2})\in {\rm H}_m}\,\mathfrak{S}\big(\mathcal{A}+\mathcal{H}\big). 
\end{align}
Let $A=\max_{a\in \mathcal A}|a|$. We claim that
\begin{align*}
\big|\mathcal I\big| \ \geq \ \lceil m/\ell\rceil,\qquad \text{where}\quad \mathcal I :=\big\{0\leq i\leq 2k-1:\,\, \exists i'>i:\,\, h_i-h_{i'}\in [-A,A]\big\}.
\end{align*}
To show this, first note that all the sums $a_j+h_i$, $j\leq \ell$, $i\not \in \mathcal{I}$ are distinct, since if there was a coincidence $a+h_i=a'+h_{i'}$ for some $a,a'\in \mathcal A$, $i'>i$, then $h_i-h_{i'}=a'-a\in [-A,A]$ and so $i\in \mathcal{I}$. Thus we must have
\begin{align*}
2k\ell-m=|\mathcal{A}+\mathcal{H}|\geq (2k-|\mathcal{I}|)\ell,
\end{align*}
so $|\mathcal{I}|\geq m/\ell$, as claimed. Now, for $I:=|\mathcal{I}|$, we have a system of $I+1$ linear inequalities constraining the vector $(h_1,\ldots, h_{2k-1})\in [-H,H]^{2k-1}$:
\begin{align*}
\begin{cases}
h_1+\cdots+h_{k-1}-h_k-\cdots-h_{2k-2}-h_{2k-1}=0\\
|h_i-h_{\sigma(i)}|\leq A,\quad i\in \mathcal{I},
\end{cases}
\end{align*}
where $h_0=0$, and $\sigma(i)\in [i+1,2k-1]$ are some integers.

Let $I'=I+1$ if $I\leq k-1$, and let $I'=I$ if $I\geq k$. Then by basic linear algebra the set of $I+1$ linear forms above contains a subset of $I'$ linearly independent forms, so there exists a vector ${\bf u}=(u_1,\ldots, u_{2k-1-I'})\in [-H,H]^{2k-1-I'}$ (depending on $h_i$) and linear forms $L_i:\mathbb{Z}^{2k-2-I'}\to \mathbb{Z}$ with bounded coefficients (and with $L_{0}\equiv 0$) such that denoting $L_{i,j}({\bf u}) = a_j + L_i({\bf u})$ we have
\begin{align}
\mathcal{H}&=\{L_0({\bf u})+O(1),\ldots, L_{2k-1}({\bf u})+O(1)\} \nonumber\\
\mathcal{A}+\mathcal{H}&= \{L_{i,j}({\bf u})+O(1): 0\leq i\le 2k-1, j\le \ell\}.
\end{align}
Thus by Corollary~\ref{cor:singular} and the fact that $I'\geq \lceil m/\ell\rceil+1_{m\leq (k-1)\ell}$, for each $m$ we have
\begin{align}\label{eq:AplusHtoLij}\begin{split}
\sum_{(h_1,\ldots,h_{2k-2})\in {\rm H}_m}\,\mathfrak{S}\big(\mathcal{A}+\mathcal{H}\big) &\le \sum_{{\bf u}\in[-H,H]^{2k-2-\lceil m/\ell\rceil+\1_{m> (k-1)\ell}}}\mathfrak{S}\big(\big(L_{i,j}({\bf u})+r\big)_{0\leq i\leq 2k-1, j\leq \ell,r=O(1)}\big)\\
&\ll H^{2k-2-\lceil m/\ell\rceil+\1_{m> (k-1)\ell}}.
\end{split}
\end{align}
Hence the bound in \eqref{m6} becomes
\begin{align*}
M_{2k} &\ll XH\sum_{m=0}^{(2k-1)\ell}(\log X)^mH^{2k-2-\lceil m/\ell\rceil+\1_{m> (k-1)\ell}}\\
&\ll X(H^{2k-1}+H^{k}(\log X)^{k\ell}+H(\log X)^{(2k-1)\ell}).
\end{align*}
Finally, the assumption $H \ge (\log X)^{\ell k/(k-1)}$ implies $M_{2k} \ll XH^{2k-1}$ as claimed.
\end{proof}

\subsection{Correlations of primes and integers having typical factorization}

In this subsection, we will prove  Lemma~\ref{le:nairtenenbaum}, which upper bounds the correlations of the von Mangoldt function and the indicator of numbers having prime factors in prescribed intervals. 

We first consider a class of multiplicative functions with `moderate growth'.

\begin{definition}\label{def:multclass}
Denote the set $\mathcal M$ of multiplicative $f:\N\to\R_{\ge0}$ with
\begin{itemize}
\item[(i)] $f(p^\nu) \le 2^\nu$ for all primes $p$, $\nu\ge1$.
\item[(ii)] for all $\eps>0$ there exists $B=B(\eps)$ such that $f(n) \le Bn^{\eps}$ for all $n\ge1$.
\end{itemize}
\end{definition}

Now we give a special case of a general bound of Henriot~\cite{henriot} for the class $\mathcal M$. Henriot's result refines earlier work of Nair--Tenenbaum~\cite{NairTen}, and importantly, it is uniform in the discriminant.

\begin{lemma}\label{lem:Henriot}
Given $k\ge1$ and a tuple $\mathcal H=\{h_1,\ldots,h_k\}\subset [1,X]$, denote $\nu_p(\mathcal{H}) = |\{h\,(\textnormal{mod }p) : h\in\mathcal H\}|$. Then for any multiplicative functions $f_j\in \mathcal M$,
\begin{align}\label{eq:Henriot}
\sum_{\sqrt{X} \le n\le X}\prod_{j=1}^k f_j(n+h_j) \ \ll_{k} \Delta_D\,X\prod_{p\le \sqrt{X}}\Big(1-\frac{\nu_p(\mathcal{H})}{p}\Big)\prod_{j=1}^k\sum_{n\le \sqrt{X}}\frac{f_j(n)}{n}
\end{align}
where $D=D(\mathcal H)=\prod_{i<j}(h_j-h_i)^2$, and
\begin{align*}
\Delta_D = \prod_{p\mid D}\bigg(1\ + \sum_{\substack{0\leq \nu_1,\ldots, \nu_k\leq 1\\(\nu_1,\ldots, \nu_k)\neq (0,\ldots, 0)}} \frac{|\{n\, (\textnormal{mod }p^{2})\,:\, p^{\nu_j}\mid\mid  n+h_j\, \forall j \}|}{p^2}\prod_{j=1}^k f_j(p)\bigg)
\end{align*}
In particular, if $|f(p)|\leq 1$ for all $p$, we have $\Delta_D \le \prod_{p\mid D}(1+ 2^k/p)$.
\end{lemma}
\begin{proof}
This is~\cite[Theorem 3]{henriot} in the special case of $x=\sqrt{X}$, $y=X$, $\delta = 1/(2k)$, with linear polynomials $Q_j(n)=n+h_j$, $Q(n) = \prod_{j=1}^k Q_j(n)$, and the multiplicative function $F(n_1,\ldots,n_k)=\prod_{j=1}^k f_j(n_j)$. Note that the discriminant of the polynomial $Q$ is $\prod_{i<j}(h_i-h_j)^2=D$ and the  sum of coefficients is $\|Q\| \ll \prod_{j=1}^k h_j \ll X^{k}$.
\end{proof}
We remark that the bound \eqref{eq:Henriot} is of the correct order of magnitude when the functions $f_j$ are not too small, e.g. $f_j(n) \ge \eta^{\Omega(n)}$ for some $\eta>0$; see~\cite[Theorem 6]{henriot}.

As mentioned, we will need in Section~\ref{sec:main} a simple upper bound for the correlations of the primes and integers with prescribed factorization patterns.

\begin{lemma}\label{le:nairtenenbaum} Let $\ell\geq 1$ be fixed and let $a_1,\ldots, a_{\ell}$ be fixed and distinct. Let $X\geq 2$ and $1\leq h\leq X$ with $h\neq a_j$ for all $1\leq j\leq \ell$. Let $\mathcal{I}\subset [1,X]$, and let $\mathcal{N}$ be the set of positive integers having no prime factors from $\mathcal{I}$. Then we have
\begin{align}\label{eq:lambdabound}
\sum_{n\leq X}\1_{\mathcal{N}}(n+h)\Lambda(n+a_1)\cdots \Lambda(n+a_{\ell})\ll X\prod_{j=1}^{\ell}\left(\frac{|h-a_j|}{\varphi(|h-a_j|)} \right)^{2^{\ell+1}}\prod_{p\in \mathcal{I}}\left(1-\frac{1}{p}\right).    
\end{align}
\end{lemma}

\begin{proof}
Let $P(z)=\prod_{p< z} p$. We may assume that the summation in \eqref{eq:lambdabound} runs over $(2X)^{1/2}\leq n\leq X$, since the contribution of $n<(2X)^{1/2}$ is negligible. Therefore, the proof of \eqref{eq:lambdabound} has been reduced to showing
\begin{align}\label{eq:lambdabound2}
\sum_{n\leq X}\1_{\mathcal{N}}(n+h)\prod_{j=1}^{\ell}\1_{(n+a_j,P((2X)^{1/2}))=1}\ll X(\log X)^{-\ell}\prod_{j=1}^{\ell}\left(\frac{|h-a_j|}{\varphi(|h-a_j|)} \right)^{2^{\ell+1}}\prod_{p\in \mathcal{I}}\left(1-\frac{1}{p}\right).    
\end{align}
Note that all the indicator functions above are multiplicative. Denote $D=\prod_{1\le i<j\le \ell}(a_i-a_j)\prod_{j\leq \ell}(h-a_j)$. Now by Lemma~\ref{lem:Henriot}, we have
\begin{align}
\sum_{n\leq X}\1_{\mathcal{N}}(n+h)\prod_{j=1}^{\ell} & \1_{(n+a_j,P((2X)^{1/2}))=1} \nonumber\\
&\ll \ \Delta_D X\prod_{p\leq X}\left(1-\frac{\ell+1}{p}\right) \sum_{n_0,n_1,\ldots, n_{\ell}\leq X}\frac{\1_{\mathcal{N}}(n_0)\prod_{j=1}^{\ell}\1_{(n_j,P((2X)^{1/2}))=1}}{n_0n_1\cdots n_{\ell}},   
\end{align}
where $\Delta_D$ is crudely bounded as
\begin{align*}
\Delta_D\le \prod_{p\mid D}\left(1+\frac{2^{\ell+1}}{p}\right)\ll \prod_{j=1}^{\ell}\left(\frac{|h-a_j|}{\varphi(|h-a_j|)} \right)^{2^{\ell+1}}.   
\end{align*}
Now \eqref{eq:lambdabound2} follows by noting that by Euler products and Mertens's theorem 
\begin{align*}
 \sum_{n\leq X}\frac{\1_{\mathcal{N}}(n)}{n}\leq \prod_{\substack{p\leq X\\p\not \in \mathcal{I}}}\left(1+\frac{1}{p}+\frac{1}{p^2}+\cdots\right)&= \prod_{p\leq X}\left(1+\frac{1}{p}+\frac{1}{p^2}+\cdots\right)\prod_{p\in \mathcal{I}}\left(1-\frac{1}{p}\right)\\
 &\ll (\log X)^{-\ell}
 \prod_{p\leq X}\left(1-\frac{\ell+1}{p}\right)^{-1}\prod_{p\in \mathcal{I}}\left(1-\frac{1}{p}\right),  
\end{align*}
and also by Mertens's theorem
\begin{equation*}
\sum_{n\leq X}\frac{\1_{(n,P((2X)^{1/2}))=1}}{n}=1+\sum_{(2X)^{1/2}\leq p\leq X} \frac{1}{p}\asymp 1.
\qedhere
\end{equation*}
\end{proof}

\section{A Fourier-analytic argument}\label{sec:fourier}

Our main theorems will be a consequence of the propositions in the previous section and Proposition~\ref{thm:fourier} below on correlations on average under suitable hypotheses. We begin by formulating the necessary hypotheses. 

In the rest of the paper, 
given a function $f:\mathbb{N}\to \mathbb{C}$, let $E_f$ be a convenient upper bound for the average of $f$, so that
\begin{align*}
\frac{1}{X}\sum_{n\leq X}|f(n)|  \le E_f(X).
\end{align*}
For example, we may simply take $E_f=O(1)$ if $f$ is bounded, or if $f=\Lambda$. 

\begin{definition}
Let $X,H,C,p>1$ and $\eta>0$, and let $f:\mathbb{N}\to \mathbb{C}$.  We say that $f$ satisfies hypothesis $\mathsf{H}_1(X,H,C,p)$ if
    \begin{align*}
    \int_{0}^{X}\int_{0}^{1}\bigg|\sum_{x\leq n\leq x+2H}f(n)e(\alpha n)\bigg|^p\, \dd{\alpha} \,\dd{x}\leq C H^{p-1}XE_f(X)^p.      \end{align*}
    
We say that $f$ satisfies hypothesis $\mathsf{H}_2(X,H,\eta)$ if
    \begin{align*}
     \sup_{\alpha}\int_{0}^{X}\bigg|\sum_{x\leq n\leq x+2H}f(n)e(\alpha n)\bigg|^2\, \dd{x}\leq \eta H^2XE_f(X)^2.  
    \end{align*}
\end{definition}

\begin{proposition} \label{thm:fourier}
Let $X\geq H\geq 2$. Let $C\geq 1$, $p\geq 2$ and $\eta>0$. 
Let $f:\mathbb{N}\to \mathbb{C}$ be any function supported on $[0,X]$ and satisfying $\mathsf{H}_2(X,H,\eta)$, and let $g:\mathbb{N}\to \mathbb{C}$ be any function satisfying $\mathsf{H}_1(X,H,C,p)$. Then we have 
\begin{align*}
\sum_{h\leq H}\bigg|\sum_{n\leq X}f(n+h)g(n)\bigg|\ll (\eta C)^{1/p}HX E_{|f|+|f|^2}(X)E_{g}(X). 
\end{align*}

\end{proposition}

Note that if $|f(n)|\leq 1$, we can simply take $E_{|f|+|f|^2}(X)=O(E_{|f|}(X))$.

\begin{proof} Let us denote
\begin{align*}
S_{f,g} := \sum_{h\le H}\Big|\sum_{n\le X}f(n+h)g(n)\Big|. 
\end{align*}
Then the task is to show that if $f$ satisfies $\mathsf{H}_2(X,H,\eta)$ and $g$ satisfies $\mathsf{H}_1(X,H,C,p)$, then
\begin{align*}
S_{f,g}\ll \eta^{1/p}C^{1/p}HXE_{|f|+|f|^2}(X)E_g(X).    
\end{align*}

For proving this, we first note that
\begin{align*}
S_{f,g}\leq \frac{1}{H}\sum_{h\le 2H}(2H-h)\Big|\sum_{n\le X}f(n+h)g(n)\Big|.
\end{align*}
We introduce unimodular coefficients $c(h)$ to denote the phase of $\sum_{n\le X}f(n+h)g(n)$, so that
\begin{align*}
|S_{f,g}| & \le \frac{1}{H}\sum_{h\le 2H}(2H-h)c(h)\sum_{n\le X}f(n+h)g(n)\\
& = \frac{1}{H}\sum_{h\le 2H}c(h)\sum_{n\le X+2H} f(n) \sum_{m\le X}g(m)\1_{n=m+h}\cdot\int_0^X \1_{x\le n,m\le x+2H}\dd{x}\\
& = \frac{1}{H}\int_0^X\int_0^1\sum_{h\le 2H}c(h)e(h\alpha) \sum_{x\le n,m\le x+2H}f(n)g(m)e\big((m-n)\alpha\big)\dd{\alpha}\dd{x},
\end{align*}
where we used the orthogonality relation $\1_{n=0}=\int_0^1 e(n\alpha)\dd{\alpha}$. Therefore, we have the upper bound
\begin{align}\label{eq:SfFourier}
|S_{f,g}| \ \leq \frac{1}{H}\left|\int_0^X\int_0^1 C_0(\alpha) F_x(-\alpha) G_x(\alpha)\dd{\alpha}\dd{x}\right|,
\end{align}
where in the triple convolution integral the three exponential sums are
\begin{align*}
C_0(\alpha) &:= \sum_{h\le 2H}c(h)e(h\alpha),\\
F_x(\alpha) &:=\sum_{x\le n\le x+2H}f(n)e(n\alpha),\\
G_x(\alpha) &:=\sum_{x\le m\le x+2H}g(m)e(m\alpha).
\end{align*}
By Fubini's theorem, the triangle inequality and Cauchy--Schwarz, from \eqref{eq:SfFourier} we deduce that
\begin{align*}
|S_{f,g}|\leq \frac{1}{H}\int_{0}^{1}|C_0(\alpha)|\left(\int_{0}^{X}|F_x(-\alpha)|^2\dd{x}\right)^{\frac{1}{2}}  \left(\int_{0}^{X}|G_x(\alpha)|^2\dd{x}\right)^{\frac{1}{2}}  \dd{\alpha}. 
\end{align*}
Let $q$ satisfy $1/p+1/q=1/2$. By taking supremum of the integral with $F_x(-\alpha)$ to the power $1/p$ and applying H\"older's inequality with exponents $(2,q,p)$, we obtain
\begin{align*}
|S_{f,g}|&\leq \frac{1}{H}\sup_{\alpha}\left(\int_{0}^{X}|F_x(-\alpha)|^2\dd{x}\right)^{\frac{1}{p}}   \left(\int_{0}^{1}|C_0(\alpha)|^2\dd{\alpha}\right)^{\frac{1}{2}}\\
& \ \cdot\left(\int_{0}^1\int_{0}^{X}|F_x(-\alpha)|^2\dd{x}\dd{\alpha}\right)^{\frac{1}{q}}\left(\int_{0}^1\left(\int_{0}^{X}|G_x(\alpha)|^2\dd{x}\right)^{\frac{p}{2}}\dd{\alpha}\right)^{\frac{1}{p}}
 =:\frac{1}{H}I_0^{1/p}I_1^{1/2}I_2^{1/q}I_3^{1/p},
\end{align*}
say. Since $f$ satisfies $\mathsf{H}_2(X,H,\eta)$, we have
\begin{align*}
I_0\leq \eta H^2XE_f(X)^2.     
\end{align*}
By  Parseval, we have
\begin{align*}
I_1\leq 2H, \quad I_2\leq 2H\sum_{n\leq X+2H}|f(n)|^2\leq 2HXE_{|f|^2}(X).    
\end{align*}
Lastly, using H\"older's inequality and the fact that $g$ satisfies $\mathsf{H}_1(X,H,C,p)$, we have 
\begin{align*}
I_3\leq X^{p/2-1} \int_{0}^1\int_{0}^{X}|G_x(\alpha)|^p\dd{x}\dd{\alpha}  \leq CH^{p-1}X^{p/2}E_g(X)^p.
\end{align*}
Combining the bounds, we see that
\begin{align*}
|S_{f,g}|&\ll H^{-1}\big(\eta H^2XE_f(X)^2\big)^{1/p}H^{1/2} \big(HXE_{|f|^2}(X)^2\big)^{1/q}   \big(CH^{p-1}X^{p/2}E_g(X)^p\big)^{1/p}\\
&\ll (\eta C)^{1/p}HXE_{|f|+|f|^2}(X)E_g(X), 
\end{align*}
as desired.
\end{proof}

\section{Proofs of the main theorems}\label{sec:main}

We are now ready to prove Theorem~\ref{thm:multsums}. As already discussed in Remark~\ref{rem:quant}, Theorem~\ref{thm_hybrid}(i) concerning the M\"obius case follows immediately as a special case. As we shall see, also Theorem~\ref{thm_hybrid}(ii) follows with essentially the same proof. 

\begin{proof}[Proof of Theorem~\ref{thm:multsums}]
Fix $\eps\in (0,1)$ and $k,\ell\geq 1$, and let $2\leq (\log X)^{\ell+\eps}\leq H\leq \exp((\log X)^{1/1000})$. Let 
\begin{align*}
\delta=10^4\,\frac{\ell}{\eps}\frac{\log \log H}{\log H}.
\end{align*}
Set
\begin{align*}
 f(n)=f_1(n) \1_{\mathcal{S}}(n)\1_{[0,X]}(n),\quad\textnormal{and}\quad g(n)=\prod_{i=2}^kf_i(n+h_i)\prod_{j=1}^{\ell}\Lambda(n+a_j),   
\end{align*}
where $\mathcal{S}$ is the set of positive integers having a prime factor in each of the intervals $[P_i, Q_i]$, with $P_i,Q_i$ as in Proposition~\ref{thm:MRTChow} and $\delta$ as above. Then by Lemma~\ref{le:nairtenenbaum} (and the fact that $\mathcal{S}^c\subset \bigcup_{j\leq J}\mathcal{S}_j^c$, where $\mathcal{S}_j^c$ is the set of integers having no prime factors in $[P_j,Q_j]$) for all $X\geq 2$ we have 
\begin{align}\label{eq:f1gcor}
\bigg|\sum_{\substack{n\leq X\\n+h\notin \mathcal{S}}} f_1(n+h) g(n)\bigg| &\le
\sum_{n\le X}\1_{\mathcal S^c}(n+h) \prod_{j=1}^{\ell}\Lambda(n+a_j) \nonumber\\
&\ll \sum_{j\leq J}X\prod_{p\in [P_j,Q_j]}\left(1-\frac{1}{p}\right) \prod_{j=1}^{\ell}\left(\frac{|h-a_j|}{\varphi(|h-a_j|)}\right)^{2^{\ell+1}}.
\end{align}

By \eqref{eq_Sbound} we have
\begin{align*}
\sum_{j\leq J}\prod_{p\in [P_j,Q_j]}\left(1-\frac{1}{p}\right) \ll \sum_{j\leq J} \frac{\log P_j}{\log Q_j}\ll\frac{\log P_1}{\log Q_1}\ll \delta+\exp(-M(f;X,Q)/2000) =: \delta'.
\end{align*}
Also by Lemma~\ref{lemma_eulerphi}, we have
\begin{align*}
\sum_{h\leq H}\prod_{j=1}^{\ell}\left(\frac{|h-a_j|}{\varphi(|h-a_j|)}\right)^{2^{\ell+1}}\ll H,
\end{align*}
and so plugging back into \eqref{eq:f1gcor} gives
\begin{align}\label{eq_HXS}
\sum_{h\leq H}\bigg|\sum_{\substack{n\leq X\\n+h\notin \mathcal{S}}} f_1(n+h)g(n)\bigg| \ \ll \ \delta' \,XH.
\end{align}

Let 
\begin{align*}
p:=2+2\lceil\ell/\eps\rceil.
\end{align*}
In view of \eqref{eq_HXS}, to prove Theorem~\ref{thm:multsums} it suffices to show that
\begin{align}\label{eq:m1Sph}
\sum_{h\leq H}\bigg|\sum_{n\leq X}f(n+h)g(n)\bigg|\ll \left(\frac{\log \log H}{\log H}+\exp(-M(f;X,Q)/(2000p))\right)HX,
\end{align}
since we can crudely estimate $2000p<10000\ell/\varepsilon$. 
Note that by Lemma~\ref{Selbtuple}, we can take
\begin{align*}
E_{g}(X)\asymp 1,    
\end{align*}
and we can trivially take $E_{f}(X), E_{|f|+|f|^2}(X)\asymp 1$, so it suffices to prove \eqref{eq:m1Sph} with an extra factor of $E_{|f|+|f|^2}(X)E_{g}(X)$ on the right-hand side.

 By Proposition~\ref{lem:Lambda6moment}  (and the fact that $p$ is even and satisfies $\ell+\eps > \ell p/(p-2)$), there exists a constant $B=B_{\ell,\eps}>0$ for which
\begin{align*}
\int_{0}^{1}\int_{0}^X\bigg|\sum_{x\leq n\leq x+2H}e(\alpha n)\prod_{i=2}^kf_i(n+h_i)\prod_{j=1}^{\ell}\Lambda(n+a_j)\bigg|^p  \dd x\dd\alpha \ \le \ BXH^{p-1}.
\end{align*}
Therefore, $g$ satisfies hypothesis $\mathsf{H}_1(X,H,B',p)$ for some $B'\ll 1$.

We also note that by Proposition~\ref{thm:MRTChow} $f$ satisfies $\mathsf{H}_2(X,H,\eta)$ with $$\eta\ll  H^{-\delta}+\exp(-M(f;X,Q)/2000).$$ 
Now, applying  Proposition~\ref{thm:fourier} with the choices of $f,g$ above, we obtain
\begin{align}\begin{split}\label{eq:final}
\sum_{h\leq H}\bigg|\sum_{n\leq X}f_1\1_{\mathcal{S}}(n+h) &\prod_{i=2}^kf_i(n+h_i)\prod_{j=1}^{\ell}\Lambda(n+a_j)\bigg|\\
&\ll \eta^{1/p}HX \ \ll \ (H^{-\delta/p}+\exp(-M(f;X,Q)/(2000p)))HX.
\end{split}
\end{align}
Thus since $\delta= 10^4\ell\varepsilon^{-1}(\log \log H)/(\log H)$, this gives \eqref{eq:m1Sph}, as desired. The proof of Theorem~\ref{thm:multsums} (and hence of Theorem~\ref{thm_hybrid}(i)) is now complete
\end{proof}

\begin{proof}[Proof of Theorem~\ref{thm_hybrid}(ii).] Let $\delta$ be as in the theorem, so that in particular $\delta\leq c/C$ with $c>0$ be small enough. Take $c=1/10000$. Applying \eqref{eq:final} and Markov's inequality we see that
\begin{align*}
\bigg|\sum_{n\leq X}\mu\1_{\mathcal{S}}(n+h) &\prod_{i=2}^k\mu(n+h_i)\prod_{j=1}^{\ell}\Lambda(n+a_j)\bigg|\\
&\ll  (H^{-\delta/(2p)}+H^{\delta/(2p)}\exp(-M(\mu;X,Q)/(2000p)))HX    
\end{align*}
for all but $\ll H^{1-\delta/(2p)}$ integers $h\leq H$. By \eqref{eq:dist} we have $M(\mu;X,Q)\geq \frac{1}{4}\log \log X+O(1)$, so we see that
\begin{align*}
H^{\delta/(2p)}\exp(-M(\mu;X,Q)/(2000p))\ll H^{-\delta/(4p)} \ll \delta,   
\end{align*}
and the claim follows. 
\end{proof}

\bibliographystyle{amsplain}

\end{document}